\documentclass[11pt]{article}
\usepackage[margin=1in]{geometry}
\usepackage{amsmath,amssymb,amsthm,mathtools,amsfonts}
\usepackage[colorlinks=true,linkcolor=blue,citecolor=blue,urlcolor=blue]{hyperref}
\usepackage{authblk}
\usepackage[sort]{cite}
\usepackage{tikz}
\usetikzlibrary{backgrounds,fit,positioning}
\usepackage{graphicx}

\allowdisplaybreaks
\numberwithin{equation}{section}

\newtheorem{theorem}{Theorem}
\newtheorem{lemma}[theorem]{Lemma}

\newtheorem{question}[theorem]{Question}
\newtheorem{proposition}[theorem]{Proposition}
\theoremstyle{remark}

\newcommand \aln[2]
{
\begin{align}\label{#1}
#2
\end{align}
}

\newcommand{\Pl}{P_{\ell}}
\newcommand{\N}{\mathbb{N}}

\begin{document}

	\title{Non-persistence of equality between chromatic polynomials and list-color functions}
	
\author[1]{Meiqiao Zhang\thanks{Corresponding author. Email: 
			meiqiaozhang95@163.com and meiqiaozhang@xmu.edu.cn.}}
\author[2]{Fengming Dong\thanks{Email: fengming.dong@nie.edu.sg 
		and donggraph@163.com.}}
	
	\affil[1]{\small School of Mathematical Sciences, Xiamen  University, China}
	
	\affil[2]{\small
		National Institute of Education,
		Nanyang Technological University, 
		Singapore}

	\date{}
	
	\maketitle

\begin{abstract}                              
 For any graph $G$, let $P(G,k)$ and $P_{\ell}(G,k)$ denote the chromatic polynomial and the list-color function of $G$, respectively. 
 It remains an open problem
 whether,  for every graph
 $G$ and integer $k$,
 the equality 
 $P(G,k)=P_{\ell}(G,k)>0$
implies that 
 $P(G,k+1)=P_{\ell}(G,k+1)$
 also holds.
 In this paper, we answer this question in the negative. For every integer $k\ge 3$, we construct an infinite family of graphs $G$ such that $P(G,k)=P_{\ell}(G,k)>0$ while $P(G,k+1)>P_{\ell}(G,k+1)$.
 Moreover, using this infinite family of graphs as attachment gadgets, we further show that any graph $H$ with 
 $P(H,k)=P_{\ell}(H,k)>0$ can be developed into an infinite family of graphs $H'$ with $P(H',k)=P_{\ell}(H',k)>0$ and $P(H',k+1)>P_{\ell}(H',k+1)$.

\end{abstract}

\section{Introduction}\label{secintro}

In this article, we consider simple graphs only. 
For any graph $G$, let $V(G)$ and $E(G)$ be the vertex set and the edge set of $G$, respectively. For any non-empty subset $V_0$ of $V(G)$, let $G[V_0]$ denote the subgraph of $G$ induced by $V_0$. 
Denote by $\N$ the set of positive integers. For any $k,r\in\N$, let $[k]=\{1,\dots,k\}$ and $(k)_r=k(k-1)\cdots(k-r+1)$.

For any graph $G$, a 
{\it proper coloring} of $G$ is a mapping $\theta:V(G)\rightarrow \N$, such that $\theta(u)\neq \theta(v)$ for all $uv\in E(G)$.
For any $k\in \N$, a 
{\it proper $k$-coloring} of $G$ is a proper coloring $\theta$ with $\theta(v)\in [k]$ for all $v\in V(G)$. 
Then the \textit{chromatic polynomial} $P(G, k)$ of $G$ is a polynomial that counts the number of proper $k$-colorings of $G$ for each $k \in \N$. We say $G$ is \textit{uniquely $k$-colorable} if $G$ has no proper $(k-1)$-colorings and $V(G)$ has a unique partition into $k$ nonempty independent sets, which implies that $P(G,k)=k!$.
The chromatic polynomial was originally designed by Birkhoff in~\cite{Birk1912} as a tool to attack the Four-Color Conjecture, but later gained unique research significance because of its elegant properties. See~\cite{Dong2021, Dong2005, Read1988, Royle2009} for reference.

To generalize proper coloring, Vizing~\cite{Vizing1976} and Erd\H{o}s, Rubin and Taylor~\cite{Erdos1979} independently introduced the notion of list-coloring. For any graph $G$, a \textit{list assignment} $L$ of $G$ is a mapping from $V(G)$ to the power set of $\N$, and an $L$-\textit{coloring} of $G$ is a proper coloring $\theta$ with $\theta(v)\in L(v)$ for all $v\in V(G)$. Denote the number of $L$-colorings of $G$ by $P(G, L)$. Then $G$ is called  \textit{uniquely $L$-colorable} if $P(G,L)=1$.

We say $L$ is a \textit{$k$-list assignment} of $G$ if $|L(v)|=k$ holds for all $v\in V(G)$. Then the \textit{list-color function} $\Pl(G,k)$ of $G$ is defined to
be the minimum value of $P(G, L)$  over all $k$-list assignments $L$ of $G$ for each $k\in\N$. 
Observe that there is a \textit{trivial $k$-list assignment} $L^*$ of $G$ with $L^*(v)=[k]$ for all $v\in V(G)$, which leads to the equality $P(G,L^*)=P(G,k)$. Then by the definition of the list-color function, for each $k\in \N$,
\aln{listcolor}
{
\Pl(G,k) \le P(G,k).
}

Note that the inequality of (\ref{listcolor}) can hold nontrivially, e.g., $\Pl(K_{2,4},2) =0<2= P(K_{2,4},2)$.
Consequently, Kostochka and Sidorenko~\cite{Kosto} asked the question for which $k$ the equality of (\ref{listcolor}) holds.
Very soon, Donner~\cite{Donner1992} answered this question by establishing a deletion-contraction formula for the list-color function.

	\begin{theorem}[\cite{Donner1992}]\label{PLCG-Donner}
For any graph $G$, $P(G,k)=\Pl(G,k)$ if $k$ is sufficiently large.
	\end{theorem}

In fact, Theorem~\ref{PLCG-Donner} reveals that the list-color function  and chromatic polynomial of every graph eventually agrees exactly. It further indicates that $\Pl(G,k)$ inherits all the nice properties of the chromatic polynomial when $k$ is sufficiently large.
Following Theorem~\ref{PLCG-Donner}, the next question is naturally that for any graph $G$, what is the minimum integer $\tau(G)$ such that $P(G,k)=\Pl(G,k)$ whenever $k\ge \tau(G)$.
 In 2009, Thomassen~\cite{Thomassen2009} gave the first answer regarding the order of the graph.

	\begin{theorem}[\cite{Thomassen2009}]\label{PLCG-Thom}
For any graph $G$, $\tau(G)\le |V(G)|^{10}+1$.
	\end{theorem}

Later in 2017, Wang, Qian and Yan~\cite{Wang2017} improved Theorem~\ref{PLCG-Thom} by developing a list-coloring version of Whitney's broken cycle theorem.

	\begin{theorem}[\cite{Wang2017}]\label{PLCG-wangg}
For any graph $G$, $\tau(G)\le\frac{|E(G)|-1}{\log(1+\sqrt{2})}+1$.
	\end{theorem}
	
	Recently, by refining the techniques in~\cite{Wang2017}, Dong and Zhang gave a better result in~\cite{Dong22}.

\begin{theorem}[\cite{Dong22}]\label{PLCG-cor1-1}
For any graph $G$, $\tau(G)\le |E(G)|-1$.
\end{theorem}

By the definition of $\tau(G)$, the authors of \cite{Donner1992,Thomassen2009,Wang2017,Dong22} were required to verify the equality for all relatively large integers $k$, which is by no means trivial. This naturally raises the question of whether such a verification is always necessary. More specifically, the following question has recurred repeatedly over the past 10 years.

\begin{question}[\cite{Chi2026,Kaul23,Dahlberg2024,Bui2021,Kaul22,KaulMudrock2021,Beck2021,Kirov2016}]\label{ques1}
For all graphs $G$ and integers $k$, does the equality 
$P(G,k)=\Pl(G,k)>0$
always imply
$P(G,k+1)=\Pl(G,k+1)$?
\end{question}

Question~\ref{ques1} is trivially true for $k=1$. For $k=2$, all graphs $G$ satisfying $P(G,2)=\Pl(G,2)>0$ have been characterized in~\cite{Allred2025}, and each such graph has the property that $P(G,k)=\Pl(G,k)$ holds for all $k\in\mathbb{N}$. For $k\ge3$, the question remains open since 2016.

Also, a question parallel to Question~\ref{ques1} concerning the DP-color function and the chromatic polynomial was posed in~\cite{KaulMudrock2021} and subsequently resolved in~\cite{Dahlberg2024,Bui2021} by constructing an infinite family of counterexamples.

In this paper, we show that the answer to Question~\ref{ques1} is also negative by constructing an infinite family of graphs as counterexamples.

\begin{theorem}~\label{main1}
For each integer $k\ge 3$, there are infinitely many graphs $G$ such that $P(G,k)=\Pl(G,k)>0$ and
$P(G,k+1)>\Pl(G,k+1)$.
\end{theorem}

Moreover, using the graphs constructed in Theorem~\ref{main1}, we are able to provide an infinite family of counterexamples based on any graph $G$ with $P(G,k)=\Pl(G,k)>0$ when $k\ge 3$.

\def \setnp {{\mathcal G}^{np}}

\begin{theorem}~\label{main2}
For any integer $k\ge 3$ and graph $G$ with $P(G,k)=\Pl(G,k)>0$, 
there exists 
an infinite set $\setnp(G)$
	such that for each graph 
	$G'\in \setnp(G)$,
	$P(G',k)=\Pl(G',k)>0$ but
$P(G',k+1)>\Pl(G',k+1)$.
\end{theorem}

Theorems~\ref{main1} and~\ref{main2} will be proven in Sections~\ref{sec2} and~\ref{sec3}, respectively, with further discussion provided in Section~\ref{sec4}.

\section{Proof of Theorem~\ref{main1}
\label{sec2}}
In this section, we prove Theorem~\ref{main1} by providing for each integer $k\ge 3$ an infinite family of graphs $G_{k,t}$ with $\Pl(G_{k,t},k)=P(G_{k,t},k)>0$ and $\Pl(G_{k,t},k+1)<P(G_{k,t},k+1)$.

Let $k$ be a fixed integer with $k\ge 3$, and let $t\in \N$. Now we give the construction of $G_{k,t}$. Let 
$$V(G_{k,t})=
X_k\cup S\cup Y_t,
	$$ 
where 
$$
\left\{
\begin{aligned}
&X_k= \{x_i : i\in [k]\},\\
&S = \{s_i : i\in [2]\},\\
&Y_t= \{y_{i,j} : i\in [4],\, j\in [t]\},
\end{aligned}
\right.
$$
and 
\aln{}{
E(G_{k,t})=&\{x_ix_j:i,j\in [k], i\neq j\} \cup \{s_ix_j: i\in [2], j\in [k], i\neq j\} \nonumber\\
&\cup \{y_{i,j} s_r: i\in [4], j\in [t], r\in [2]\}.\nonumber
}

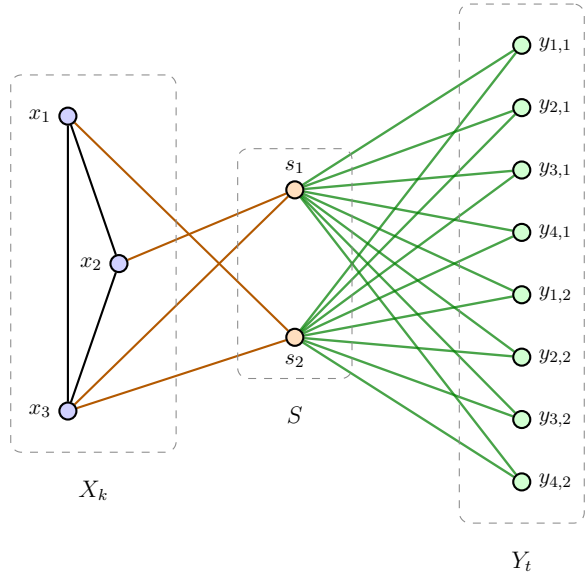
\begin{figure}[htbp]
\centering
\begin{tikzpicture}[
 scale=0.75,
    transform shape,
    x=1cm,
    y=1cm,
    vertex/.style={circle,draw=black,thick,minimum size=3mm,
        inner sep=0pt,font=\small},
    xvertex/.style={vertex,fill=blue!18},
    svertex/.style={vertex,fill=orange!25},
    yvertex/.style={vertex,fill=green!18,minimum size=3mm},
    xedge/.style={black,thick},
    sedge/.style={orange!70!black,thick},
    yedge/.style={green!50!black,line width=.9pt,opacity=.72},
    group/.style={draw=black!45,dashed,rounded corners=4pt,
        inner sep=8pt},
    grouplabel/.style={font=\bfseries\large}
]

\node[xvertex,label=left:{$x_1$}] (x1) at (0, 2.6) {};
\node[xvertex,label=left:{$x_2$}] (x2) at (0.9, 0) {};
\node[xvertex,label=left:{$x_3$}] (x3) at (0,-2.6) {};

\node[svertex,label=above:{$s_1$}] (s1) at (4, 1.3) {};
\node[svertex,label=below:{$s_2$}] (s2) at (4,-1.3) {};

\node[yvertex,label=right:{$y_{1,1}$}] (y111) at (8, 3.85) {};
\node[yvertex,label=right:{$y_{2,1}$}] (y112) at (8, 2.75) {};
\node[yvertex,label=right:{$y_{3,1}$}] (y121) at (8, 1.65) {};
\node[yvertex,label=right:{$y_{4,1}$}] (y122) at (8, 0.55) {};
\node[yvertex,label=right:{$y_{1,2}$}] (y211) at (8,-0.55) {};
\node[yvertex,label=right:{$y_{2,2}$}] (y212) at (8,-1.65) {};
\node[yvertex,label=right:{$y_{3,2}$}] (y221) at (8,-2.75) {};
\node[yvertex,label=right:{$y_{4,2}$}] (y222) at (8,-3.85) {};

\begin{scope}[on background layer]
    \draw[xedge] (x1)--(x2)--(x3)--(x1);

    \draw[sedge] (s1)--(x2);
    \draw[sedge] (s1)--(x3);
    \draw[sedge] (s2)--(x1);
    \draw[sedge] (s2)--(x3);

    \foreach \y in {y111,y112,y121,y122,y211,y212,y221,y222} {
        \draw[yedge] (s1)--(\y);
        \draw[yedge] (s2)--(\y);
    }
\end{scope}

\begin{scope}[on background layer]
    \node[group,fit=(x1)(x2)(x3), inner xsep=18pt,
        inner ysep=12pt] (Xbox) {};
    \node[group,fit=(s1)(s2),    inner xsep=18pt,
        inner ysep=12pt] (Sbox) {};
    \node[group,fit=(y111)(y222),    inner xsep=20pt,
        inner ysep=12pt] (Ybox) {};
\end{scope}
\node[grouplabel,below=10pt of Xbox] {$X_k$};
\node[grouplabel,below=10pt of Sbox] {$S$};
\node[grouplabel,below=10pt of Ybox] {$Y_t$};
\end{tikzpicture}
\caption{The graph $G_{3,2}$.}
\label{fig:H32}
\end{figure}

An example for $G_{3,2}$ is as shown in Figure~\ref{fig:H32}.
It is clear that $|V(G_{k,t})|=k+2+4t$, and $G_{k,t}[X_k]$ and all $G_{k,t}[X_k\cup \{s_i\}\setminus \{x_i\}]$ with $i\in [2]$ are isomorphic to $K_k$.

In the following, we shall prove that each $G_{k,t}$ is a graph satisfying Theorem~\ref{main1} whenever $t$ is sufficiently large.
We first show that $\Pl(G_{k,t},k)=P(G_{k,t},k)$.

\begin{lemma}\label{prop:equality-at-k}
For any $k\ge 3$ and $t\in \N$, 
\aln{ine1-1}
{
    \Pl(G_{k,t},k)=P(G_{k,t},k)=k!(k-2)^{4t}.
}
\end{lemma}

\begin{proof}
We shall prove that $P(G_{k,t},k)=k!(k-2)^{4t}$ by counting the number of proper $k$-colorings of $G_{k,t}$ in the order of $x_1,x_2,\dots, x_k, s_1,s_2,y_{1,1}, y_{1,2},\dots,y_{4,t}$. It is easy to see that we have $k-i+1$ available colors for each $x_{i}$. Since $k\ge 3$ and each $s_i$ is adjacent to $k-1$ vertices in $X_k$, we then have exactly one available color for each $s_i$, which is the color we assign to $x_i$. Moreover, in any proper $k$-coloring of $G_{k,t}$, the colors of all $x_i$ are pairwise distinct as all $x_i$ together forms a clique. Thus $s_1$  and $s_2$ have distinct colors, which leaves exactly $k-2$ available colors for each $y_{i,j}$.
Hence $P(G_{k,t},k)=k!(k-2)^{4t}$.

Now it remains to show that  $\Pl(G_{k,t},k)=k!(k-2)^{4t}$. Since $\Pl(G_{k,t},k)\le P(G_{k,t},k)$, we need only to prove that $P(G_{k,t}, L)\ge k!(k-2)^{4t}$
for all $k$-list assignments $L$ of $G_{k,t}$.

Let $L$ be any 
$k$-list assignment of $G_{k,t}$.
We shall give a lower bound for the number of $L$-colorings in the order of $x_1,x_2,\dots, x_k, s_1,s_2,y_{1,1}, y_{1,2},\dots,y_{4,t}$. 
It is easy to see that we have at least $k-i+1$ available colors for each $x_{i}$, at least one available color for each $s_i$, and at least $k-2$ colors for each $y_{i,j}$, which implies $P(G_{k,t}, L)\ge k!(k-2)^{4t}$. 
\end{proof}

Next, we establish a lower bound for $P(G_{k,t},k+1)$.

\begin{lemma}\label{lem:ordinary-lower-bound}
For any $k\ge 3$ and $t\in \N$, 
\aln{ine1-2}
{
    P(G_{k,t},k+1)\ge (k+1)!k^{4t}.
}
\end{lemma}

\begin{proof}
We shall also prove by establishing a lower bound for the number of proper $(k+1)$-colorings of $G_{k,t}$ in the order of $s_1,s_2,x_1,x_2,\dots, x_k, y_{1,1}, y_{1,2},\dots,y_{4,t}$. We can first assign a same color to both $s_1$ and $s_2$, for which we have $k+1$ choices. Then we have $k-i+1$ available colors for each $x_{i}$, and $k$ available colors for each $y_{i,j}$. Hence
(\ref{ine1-2}) holds.
\end{proof}

In the following, we shall determine an upper bound for $\Pl(G_{k,t},k+1)$ by finding a special ($k+1$)-list assignment for $G_{k,t}$.

Let $C,B_1,B_2$ be pairwise disjoint sets where $C=\{c_1,c_2,\dots,c_{k-1}\}$, $B_1=\{b_{1,1},b_{1,2}\}$ and 
$B_2=\{b_{2,1},b_{2,2}\}$. Then let $D_1=\{b_{1,1},b_{2,1}\}$, $D_2=\{b_{1,1},b_{2,2}\}$, $D_3=\{b_{1,2},b_{2,1}\}$, and $D_4=\{b_{1,2},b_{2,2}\}$.

Now let $L$ be the $(k+1)$-list assignment of $G_{k,t}$ such that
\aln{def-L}
{
\left\{
\begin{aligned}
&L(x_i) =C\cup B_1~\text{for all}~i\in [k],\\
&L(s_i) = C\cup B_i~\text{for all}~i\in [2],\\
&L(y_{i,j})= C\cup D_{i}~\text{for all}~i\in [4],\, j\in [t].
\end{aligned}
\right.
}
Then it is easy to see that $L$ is a ($k+1$)-list assignment, hence 
$
\Pl(G_{k,t},k+1)\le P(G_{k,t},L).
$
We shall further determine $P(G_{k,t},L)$ by the next two lemmas.

\begin{lemma}\label{lem:dominant-product}
Let $L|_{X_k\cup S}$ be the $(k+1)$-list assignment of $G_{k,t}[X_k\cup S]$ which is the restriction of $L$ to $X_k\cup S$. Then for each $L|_{X_k\cup S}$-coloring $\theta$ of $G_{k,t}[X_k\cup S]$, the number of $L$-colorings $c$ of $G_{k,t}$ with $c|_{X_k\cup S}=\theta$ is 
$$
\left\{
\begin{alignedat}{2}
& k^{4t}, 
&& \qquad \text{if } \theta(s_1)=\theta(s_2),\\
& (k-1)^{4t}, 
&& \qquad \text{if } \theta(s_1)\neq\theta(s_2)~\text{and}~
   \theta(s_1),\theta(s_2)\in C,\\
& \bigl(k(k-1)\bigr)^{2t}, 
&& \qquad \text{if } \theta(s_1)\neq\theta(s_2)~\text{and}~
   |\{\theta(s_1),\theta(s_2)\}\cap C|=1,\\
& \bigl((k-1)k^2(k+1)\bigr)^t, 
&& \qquad \text{if } \theta(s_1)\neq\theta(s_2)~\text{and}~
   |\{\theta(s_1),\theta(s_2)\}\cap C|=0.
\end{alignedat}
\right.
$$
\end{lemma}

\begin{proof}
Let $\theta$ be any $L|_{X_k\cup S}$-coloring of $G_{k,t}[X_k\cup S]$, and let $Z=\{\theta(s_i): i\in[2]\}$. Then $1\le |Z|\le 2$. To extend $\theta$ into an $L$-coloring of $G_{k,t}$, it remains to assign colors to the vertices in $Y_t$.
For each
$y_{i,j}\in Y_t$, we have $|(C\cup D_i)\setminus Z|$ available colors to choose. 
Then the number of $L$-colorings of $G_{k,t}$ extending $\theta$ is 
\aln{ineq91}{
\left(\prod_{i=1}^{4}\left|(C\cup D_i)\setminus Z\right|\right)^t.
}

If $|Z|=1$
(i.e., $\theta(s_1)= \theta(s_2)$), 
then $\theta(s_1)=\theta(s_2)=c_s\in C$ for some $s\in [k-1]$.
Thus $|(C\cup D_i)\setminus Z|=k$ for all $i\in [4]$, implying that 
\aln{ineq3}{
	\prod_{i=1}^{4}|(C\cup D_i)\setminus Z|
	= k^{4}.
}
By (\ref{ineq91}), 
the result holds
when $\theta(s_1)=\theta(s_2)$.

Now consider the case $|Z|=2$
(i.e., $\theta(s_1)\ne \theta(s_2)$.)
Let $q=|Z\cap C|$ and $p=|Z|-q$. Then $p$ is the number of colors in $Z\cap (B_1\cup B_2)$.
Moreover, either $(q,p)=(2,0)$, or 
$(q,p)=(1,1)$, or $(q,p)=(0,2)$. We shall analyze each of the three cases below.

\noindent{\bf Case 1}: 
$(q,p)=(2,0)$.

In this case, $\theta(s_1)\neq\theta(s_2)$ and $\theta(s_1),\theta(s_2)\in C$.
Then
$\left|(C\cup D_i)\setminus Z\right|=k-1$ for all $i\in [4]$,
implying that 
\aln{ineq3-c1}
{
	\prod_{i=1}^{4}|(C\cup D_i)\setminus Z|
= (k-1)^{4}.
}

\noindent{\bf Case 2}: 
$(q,p)=(1,1)$.

In this case, $\theta(s_1)\neq\theta(s_2)$ and $
|\{\theta(s_1),\theta(s_2)\}\cap C|=1$.
Without loss of generality,
assume that $Z=\{c_1,b_{1,1}\}$.
Then,
\aln{}
{
	\left \{
	\begin{array}{l}
		|(C\cup D_1)\setminus Z|
		=|(C\cup D_2)\setminus Z|
		=k-1;\\
		|(C\cup D_3)\setminus Z|
		=|(C\cup D_4)\setminus Z|
		=k.
	\end{array}
	\right.
}
Thus
\aln{ineq3-c2}
{
	\prod_{i=1}^{4}|(C\cup D_i)\setminus Z|
	= k^2(k-1)^2.
}

\noindent{\bf Case 3}: 
$(q,p)=(0,2)$. 

In this case, $\theta(s_1)\neq\theta(s_2)$ and $
|\{\theta(s_1),\theta(s_2)\}\cap C|=0$.
Without loss of generality,
assume that $Z=\{b_{1,1},b_{2,1}\}$.
Then,
\aln{}
{
	\left \{
	\begin{array}{l}
	|(C\cup D_1)\setminus Z|=k-1;\\
	|(C\cup D_2)\setminus Z|
	=|(C\cup D_3)\setminus Z|
	=k;\\
	|(C\cup D_4)\setminus Z|=k+1.
	\end{array}
\right.
}
Thus, 
\aln{ineq3-c3}
{
	\prod_{i=1}^{4}|(C\cup D_i)\setminus Z|
	= (k-1)k^2(k+1).
}

By (\ref{ineq91})
and the results in 
the three cases above, 
the lemma is proven.
\end{proof}

\begin{lemma}\label{lem:list-upper-bound}
For any $k\ge 3$ and $t\in \N$, 
let $L$ be the $(k+1)$-list assignment of $G_{k,t}$
defined in (\ref{def-L}). Then 
\aln{ine1-6}
{
    P(G_{k,t},L)=&
    (k-1)k!k^{4t}+3(k-2)(k-1)!(k-1)^{4t+1}+\big(4(k-1)k!\nonumber\\
    &+6(k-1)(k-1)!\big)\bigl(k(k-1)\bigr)^{2t}+8k!\bigl((k-1)k^2(k+1)\bigr)^t.
}
\end{lemma}

\begin{proof}
We shall prove by analyzing all possible $L|_{X_k\cup S}$-colorings of $G_{k,t}[X_k\cup S]$ and combine them with Lemma~\ref{lem:dominant-product}.

We first count the number of $L|_{X_k\cup S}$-colorings $\theta$ of $G_{k,t}[X_k\cup S]$ such that $\theta(s_1)=\theta(s_2)$. Suppose $\theta(s_1)=\theta(s_2)=c$. Then $c\in C$ and $c$ can be any one of the $k-1$ colors in $C$. According to the construction of $G_{k,t}$, each vertex in $X_k$ is adjacent to at least one vertex in $S$, which implies that no vertex in $X_k$ can receive color $c$. Therefore, all vertices in $X_k$ have all colors in $(C\setminus \{c\})\cup B_1$ to use, where $|(C\setminus \{c\})\cup B_1|=k$. Since $X_k$ is a clique of size $k$, we have exactly $k!$
ways to assign colors to $X_k$. Thus in this case, the number of 
$L|_{X_k\cup S}$-colorings $\theta$ of $G_{k,t}[X_k\cup S]$ such that $\theta(s_1)=\theta(s_2)$ is $(k-1)k!$. Then by Lemma~\ref{lem:dominant-product},
the number of $L$-colorings $\theta'$ of $G_{k,t}$ such that $\theta'(s_1)=\theta'(s_2)$ is  $(k-1)k! k^{4t}$.

For the remaining $L|_{X_k\cup S}$-colorings $\theta$ of
$G_{k,t}[X_k\cup S]$, $\theta(s_1)\neq\theta(s_2)$ and we shall discuss each of the three cases below.

\noindent{\bf Case 1}: $\theta(s_1)\neq\theta(s_2)$ and $\theta(s_1),\theta(s_2)\in C$.

For this case, we count by coloring the vertices in the order
$s_1,s_2,x_3,x_4,\dots,x_k,x_1,x_2$.
There are $(k-1)(k-2)$ ways to color $s_1$ and $s_2$, followed by
$(k-1)(k-2)\cdots 2$ ways to color $x_3,x_4,\dots,x_k$, which leaves exactly one remaining color in $C\cup B_1$, say $q$.
We may then assign to $(x_1,x_2)$ any of the three ordered pairs
$(\theta(s_1),q)$, $(q,\theta(s_2))$, and
$(\theta(s_1),\theta(s_2))$.
Hence, the number of such colorings is
$3(k-1)(k-2)(k-1)!$.

\noindent{\bf Case 2}: $\theta(s_1)\neq\theta(s_2)$ and $|\{\theta(s_1),\theta(s_2)\}\cap C|=1$.

Suppose $\theta(s_1)\in C$ and $\theta(s_2)\in B_2$ first. 
Then we count the colorings in the order
$s_1,s_2,x_2,x_3,\dots,x_k,x_1$.
There are $2(k-1)$ ways to color $s_1$ and $s_2$, followed by
$k(k-1)\cdots 2$ ways to color $x_2,x_3,\dots,x_k$.
This leaves exactly one remaining color in $C\cup B_1$, say $q$.
We may then assign the color $q$ or $\theta(s_1)$ to $x_1$.
Hence, the number of such colorings is
$4(k-1)k!$.

Then we suppose $\theta(s_2)\in C$ and $\theta(s_1)\in B_1$.
We count the colorings  in the order
$s_1,s_2,x_3,x_4,\dots,x_k,x_1,x_2$.
There are $2(k-1)$ ways to color $s_1$ and $s_2$, followed by
$(k-1)(k-2)\cdots 2$ ways to color $x_3,x_4,\dots,x_k$,
which leaves exactly one remaining color in $C\cup B_1$, say $q$.
We may then assign to $(x_1,x_2)$ any of the three ordered pairs
$(\theta(s_1),q)$, $(q,\theta(s_2))$, and
$(\theta(s_1),\theta(s_2))$.
Hence, the number of such colorings is
$6(k-1)(k-1)!$.

\noindent{\bf Case 3}: $\theta(s_1)\neq\theta(s_2)$ and $|\{\theta(s_1),\theta(s_2)\}\cap C|=0$.

For this case, we count the colorings  in the order
$s_1,s_2,x_2,x_3,\dots,x_k,x_1$.
There are $2\times 2$ ways to color $s_1$ and $s_2$, followed by
$k(k-1)\cdots 2$ ways to color $x_2,x_3,\dots,x_k$, which leaves exactly one remaining color in $C\cup B_1$, say $q$.
We may then assign the color $q$ or $\theta(s_1)$ to $x_1$.
Hence, the number of such colorings is
$8k!$.

By Lemma~\ref{lem:dominant-product} and the three cases above, we have
$3(k-1)(k-2)(k-1)!(k-1)^{4t}+\big(4(k-1)k!+6(k-1)(k-1)!\big)\bigl(k(k-1)\bigr)^{2t}+8k!\bigl((k-1)k^2(k+1)\bigr)^t$ 
$L$-colorings $\theta'$ of $G_{k,t}$ such that $\theta'(s_1)\neq \theta'(s_2)$.

Hence the result is proven.
\end{proof}

Now we are able to show that $\Pl(G_{k,t},k+1)<P(G_{k,t},k+1)$ when $t$ is sufficiently large.

\begin{proposition}\label{thm:main}
For any integer $k\geq 3$, there is an integer $t_0=t_0(k)$ such that
$    \Pl(G_{k,t},k+1)<P(G_{k,t},k+1)
$ holds for 
all $t\geq t_0$.
\end{proposition}

\begin{proof}
Let $L$ be the $(k+1)$-list assignment of $G_{k,t}$
	defined in (\ref{def-L}). 
Then by Lemmas~\ref{lem:ordinary-lower-bound} and~\ref{lem:list-upper-bound},
\aln{ineq8}{
&P(G_{k,t},k+1)-\Pl(G_{k,t},k+1)\nonumber\\ 
\ge & P(G_{k,t},k+1)-P(G_{k,t},L)\nonumber\\
\ge & (k+1)!k^{4t}-     (k-1)k!k^{4t}-3(k-2)(k-1)!(k-1)^{4t+1}-\big(4(k-1)k!\nonumber\\
&
+6(k-1)(k-1)!\big)\bigl(k(k-1)\bigr)^{2t}-8k!\bigl((k-1)k^2(k+1)\bigr)^t.
}
Then we need only to show that for all sufficiently large integers $t$,
\aln{ineq8-1}{
(k+1)!k^{4t}>   & (k-1)k!k^{4t}-3(k-2)(k-1)!(k-1)^{4t+1}
+(4(k-1)k!\nonumber\\
&
+6(k-1)(k-1)!)\bigl(k(k-1)\bigr)^{2t}+8k!\bigl((k-1)k^2(k+1)\bigr)^t.
}

Let $r=1-1/k, s=1-1/k^2,$ and
\begin{equation}\label{eq:def-F}
F_k(t)
=\frac{3(k-1)(k-2)}{2k}r^{4t}
+\frac{(2k+3)(k-1)}{k}r^{2t}
+4s^t.
\end{equation}
Then it can be easily verified that (\ref{ineq8-1}) holds if and only if $F_k(t)<1$.
Let $t_0$ be the integer defined below:
$$t_0=\left\lceil k^2\log\left(\frac{7k+1}{2}\right)\right\rceil.$$
In the following, we shall prove that $F_k(t)<1$ whenever $t\ge t_0$.

Since $k\ge 3$, we have
\[
r^2=\left(1-\frac1k\right)^2
<1-\frac1{k^2}=s,
\]
implying that $r^{4t}<r^{2t}<s^t$.  Then for any integer $t\ge t_0$,
\aln{ineq8-2}{
F_k(t)<&\left(
\frac{3(k-1)(k-2)}{2k}
+\frac{(2k+3)(k-1)}{k}+4
\right)s^t\nonumber\\
=&\frac{7k+1}{2}(1-\frac{1}{k^2})^t\nonumber\\
<&\frac{7k+1}{2}\exp(-t/k^2)\nonumber\\
\le&\frac{7k+1}{2}\exp\left(-\log\left(\frac{7k+1}{2}\right)\right)\nonumber\\
=1.
}
Hence the result holds.
\end{proof}

By Lemma~\ref{prop:equality-at-k} and Proposition~\ref{thm:main},  Theorem~\ref{main1} is proven.

\section{Proof of Theorem~\ref{main2}
\label{sec3}}
In this section, we prove Theorem~\ref{main2} after introducing an elementary property of the list-color function.

Let $r\in\N$. For any pair of vertex-disjoint graphs $G$ and $H$ that each contains a copy of $K_r$, let $G\cup_r H$ denote a graph obtained by identifying a copy of $K_r$ in  $G$ with a copy of $K_r$ in $H$. 
It is well known that 
for any $r,k\in \N$,
\aln{ineq11}
{
  P(G\cup_r H,k)=\frac{P(G,k)P(H,k)}{(k)_r}.
}
However, for the list-color function, the analogous identity does not hold in general. For example, $\Pl(C_4,2)=2$ while $\Pl(C_4\cup_1 C_4,2)=0$.

On the other hand, for the DP-color function, it is known that
$$P_{DP}(G\cup_r H,k)\le \frac{P_{DP}(G,k)P_{DP}(H,k)}{(k)_r}$$ holds for $r=1,2$ and all $k\in\N$,
whereas a counterexample exists for $r=3$ and $k=4$~\cite{Beck2021,KaulMax21}. This naturally raises the question of whether for the list-color function,
\aln{ineq11-1}
{\Pl(G\cup_r H,k)\le \frac{\Pl(G,k)\Pl(H,k)}{(k)_r}}
holds for all $r,k\in\N$.

In fact, (\ref{ineq11-1}) is not true even for $r=2$.
To see this, consider the graph $K_{2,12}\cup_2 K_3$. 
It is clear that for any 3-list assignment $L$ of $K_{2,12}\cup_2 K_3$, an $L|_{V(K_{2,12})}$-coloring of $K_{2,12}$ naturally extends  to an $L$-coloring of $K_{2,12}\cup_2 K_3$,
implying that $\Pl(K_{2,12}\cup_2 K_3,3)\ge \Pl(K_{2,12},3)$.
Moreover, results in~\cite{Kaul23} show that every $3$-list assignment 
$L$ of $K_{2,12}$ satisfying $P(K_{2,12},L)=\Pl(K_{2,12},3)$
has the property that $L(u)\neq L(v)$ for every 
$uv\in E(K_{2,12})$. Then for any 3-list assignment $L$ of $K_{2,12}\cup_2 K_3$ such that $P(K_{2,12},L|_{V(K_{2,12})})=\Pl(K_{2,12},3)$, an $L|_{V(K_{2,12})}$-coloring of $K_{2,12}$ does 
not extend uniquely to $L$-colorings of $K_{2,12}\cup_2 K_3$,
implying that
$\Pl(K_{2,12}\cup_2 K_3,3)>\Pl(K_{2,12},3)$.
Therefore
\[
\Pl(K_{2,12}\cup_2 K_3,3)>\Pl(K_{2,12},3)=
\frac{\Pl(K_{2,12},3)\Pl(K_3,3)}{3\cdot 2},
\]
which provides a counterexample to (\ref{ineq11-1}) when $r=2$.

Thus, we restrict our attention to the case $r=1$, as follows.
\begin{lemma}\label{lemm}
For  any $k\in \N$ and any pair of vertex-disjoint graphs $G$ and $H$,
\aln{ineq12}
{
  \Pl(G\cup_1 H,k)
  \leq \frac{\Pl(G,k)\Pl(H,k)}{k}.
}
\end{lemma}

\begin{proof}
Suppose $G\cup_1 H$ is obtained by identifying $u$ in $V(G)$ and $v$ in $V(H)$.

Let $L_G$ and $L_H$ be a $k$-list assignment of $G$ and $H$, respectively, such that 
\[
  P(G,L_G)=\Pl(G,k)
  \qquad\text{and}\qquad
  P(H,L_H)=\Pl(H,k).
\]
Moreover, we assume without loss of generality that
\[
  L_G(u)=\{1,\dots,k\}
  \qquad\text{and}\qquad
  L_H(v)=\{1,\dots,k\}.
\]

In the following, we shall construct a $k$-list assignment $L$ of $G\cup_1 H$ such that  
\aln{ine10}{
P(G\cup_1 H,L) \leq \frac{ P(G,L_G)  P(H,L_H)}{k}= \frac{\Pl(G,k)\Pl(H,k)}{k}.
}

For $i=1,2,\dots,k$, let $g_i$ be the number of $L_G$-colorings $\theta$ such that $\theta(u)=i$, and let $h_i$ be the number of $L_H$-colorings $\varphi$ such that $\varphi(v)=i$.
Then
\[
  \sum_{i=1}^k g_i=  P(G,L_G)
  \qquad\text{and}\qquad
  \sum_{j=1}^k h_j=  P(H,L_H).
\]

For each permutation $\sigma\in S_k$, set
$
  S_\sigma=\sum_{i=1}^k g_i h_{\sigma(i)}.
$
Then
\aln{eq5}{
  \frac{1}{k!}\sum_{\sigma\in S_k}S_\sigma
   &=\ \frac{1}{k!}\sum_{\sigma\in S_k}\left(\sum_{i=1}^k g_i h_{\sigma(i)}
   \right)\nonumber\\
  &=\sum_{i=1}^k g_i
    \left(\frac{1}{k!}\sum_{\sigma\in S_k}h_{\sigma(i)}\right)\nonumber\\
  &=\sum_{i=1}^k g_i
    \left(\frac{(k-1)!}{k!}\sum_{j=1}^k h_j\right)\nonumber\\
  &=\frac{1}{k}
    \left(\sum_{i=1}^k g_i\right)
    \left(\sum_{j=1}^k h_j\right)\nonumber\\
  &=\frac{  P(G,L_G)  P(H,L_H)}{k}.
}
Since $|S_k|=k!$, (\ref{eq5}) indicates a permutation $\sigma_0\in S_k$ such that
\aln{eq5-1}{
  S_{\sigma_0}
  \leq \frac{  P(G,L_G)  P(H,L_H)}{k}= \frac{\Pl(G,k)\Pl(H,k)}{k}.
}
Let $f$ be the mapping from $\N$ to $\N$ such that $f(\sigma_0(i))=i$ for all $i\in [k]$ and $f(i)=i$ for $i\ge k+1$.

Now we shall construct a special $k$-list assignment $L$ of $G\cup_1H$, where
$$
\left\{
\begin{aligned}
&L(x)= L_G(x)~\text{for all}~x\in V(G),\\
&L(y)=\{f(i): i\in L_H(y)\}~\text{for all}~y\in V(H)\setminus\{v\}.
\end{aligned}
\right.
$$
Then it is clear by (\ref{eq5-1}) that 
\aln{eq5-2}{
P_\ell(G\cup_1 H,k)\le  P(G\cup_1 H,L)
  =\sum_{i=1}^k g_i h_{\sigma_0(i)}
  =S_{\sigma_0}
    \leq \frac{\Pl(G,k)\Pl(H,k)}{k}.
}
\end{proof}

Now we give the proof of Theorem~\ref{main2}.

\noindent\textit{Proof of Theorem~\ref{main2}.}
For any graph $G$ with $P(G,k)=\Pl(G,k)>0$, let $H_{k,t}$ be a graph obtained by identifying one vertex of $G$ and $x_k$ of $G_{k,t}$. We shall show that $P(H_{k,t},k)=\Pl(H_{k,t},k)>0$ and
$P(H_{k,t},k+1)>\Pl(H_{k,t},k+1)$ whenever $t$ is sufficiently large.

Recall that $P(G_{k,t},k+1)>\Pl(G_{k,t},k+1)$ whenever $t$ is sufficiently large. Then by (\ref{ineq11}) and (\ref{ineq12}), 
\aln{}{
\Pl(H_{k,t},k+1)&\le \frac{\Pl(G,k+1)\Pl(G_{k,t},k+1)}{k+1}\nonumber\\
&<
\frac{P(G,k+1)P(G_{k,t},k+1)}{k+1}\nonumber\\
&=P(H_{k,t},k+1).
}
Thus it remains to show that $P(H_{k,t},k)=\Pl(H_{k,t},k)>0$.

By Lemma~\ref{prop:equality-at-k} and (\ref{ineq11}), we have
\aln{ieq6-11}{
P(H_{k,t},k)= \frac{P(G,k)P(G_{k,t},k)}{k}
=(k-1)!(k-2)^{4t}P(G,k)>0.
}

We next consider $\Pl(H_{k,t},k)$. 
Let $L$ be a $k$-list assignment of $H_{k,t}$ such that $\Pl(H_{k,t},k)=P(H_{k,t},L)$. Then there are at least $\Pl(G,k)$  $L|_{V(G)}$-colorings of $G$. Moreover, each of them can be extended to at least $(k-1)!(k-2)^{4t}$ 
$L$-colorings of $H_{k,t}$  by considering in the order $x_1,\ldots,x_{k-1}, s_1, s_2,y_{1,1}, \dots, y_{4,t}$. Thus
\aln{ieq6-1}{
   \Pl(H_{k,t},k)=P(H_{k,t},L)\geq (k-1)!(k-2)^{4t}\Pl(G,k)= (k-1)!(k-2)^{4t}P(G,k).
}
Hence $\Pl(H_{k,t},k)=     P(H_{k,t},k)>0$ follows from (\ref{ieq6-11}) and (\ref{ieq6-1}).
\qed

\section{Concluding Remarks
\label{sec4}}
Although all counterexamples constructed in the proofs of Theorems~\ref{main1} and~\ref{main2} for each integer $k\ge 3$ contain copies of $K_k$, we shall show in this section that the presence of $K_k$ is not essential.

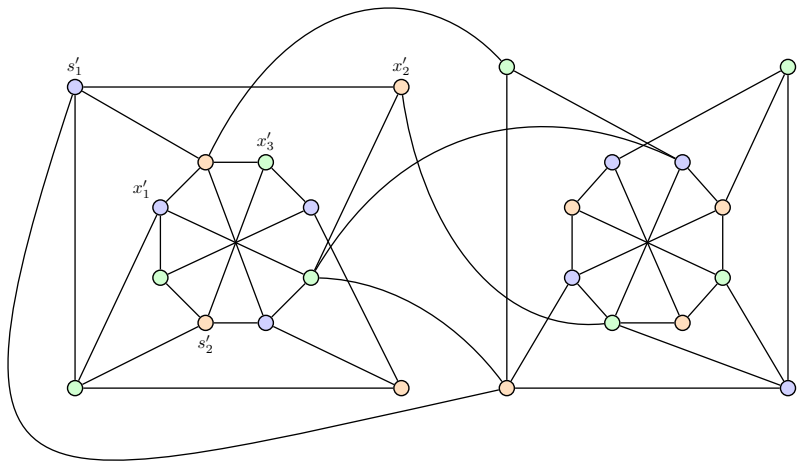
\begin{figure}[htbp]
\centering
\resizebox{0.8\textwidth}{!}{%
\begin{tikzpicture}[
    x=1cm,
    y=1cm,
    vertex/.style={
        circle,
        draw=black,
        thick,
        minimum size=3mm,
        inner sep=0pt
    },
    xvertex/.style={vertex,fill=blue!18},
    svertex/.style={vertex,fill=orange!25},
    yvertex/.style={vertex,fill=green!18,minimum size=3mm},
    edge/.style={draw=black,line width=0.75pt},
    every label/.style={text=black,font=\small,inner sep=2pt}
]

\node[yvertex,label=below left:{}]  (u1)  at (0,0) {};
\node[xvertex,label=above:{$s_1'$}]       (u2)  at (0,6) {};
\node[svertex,label=above:{$x_2'$}]       (u3)  at (6.5,6) {};
\node[svertex,label=below:{}]       (u4)  at (6.5,0) {};
\node[yvertex,label=below left:{}]  (u5)  at (1.7,2.2) {};
\node[xvertex,label=above left:{$x_1'$}]  (u6)  at (1.7,3.6) {};
\node[svertex,label=above:{}]       (u7)  at (2.6,4.5) {};
\node[yvertex,label=above:{$x_3'$}]       (u8)  at (3.8,4.5) {};
\node[xvertex,label=above right:{}] (u9)  at (4.7,3.6) {};
\node[yvertex,label=below right:{}](u10) at (4.7,2.2) {};
\node[xvertex,label=below:{}]      (u11) at (3.8,1.3) {};
\node[svertex,label=below:{$s_2'$}]      (u12) at (2.6,1.3) {};

\node[xvertex,label=below right:{}]  (p1)  at (14.2,0) {};
\node[svertex,label=below left:{}]   (p2)  at (8.6,0) {};
\node[yvertex,label=above left:{}]   (p3)  at (8.6,6.4) {};
\node[yvertex,label=above right:{}]  (p4)  at (14.2,6.4) {};
\node[svertex,label=below right:{}]  (p5)  at (12.1,1.3) {};
\node[yvertex,label=below left:{}]   (p6)  at (10.7,1.3) {};
\node[xvertex,label=left:{}]         (p7)  at (9.9,2.2) {};
\node[svertex,label=left:{}]         (p8)  at (9.9,3.6) {};
\node[xvertex,label=above left:{}]   (p9)  at (10.7,4.5) {};
\node[xvertex,label=above right:{}] (p10) at (12.1,4.5) {};
\node[svertex,label=right:{}]       (p11) at (12.9,3.6) {};
\node[yvertex,label=right:{}]       (p12) at (12.9,2.2) {};

\begin{scope}[on background layer]
    \draw[edge] (u7)  .. controls (4.4,8.15) and (7.0,8.15) .. (p3);
    \draw[edge] (u2)  .. controls (-3,-3) and (-1.15,-2.15) .. (p2);
    \draw[edge] (u10) .. controls (6.45,2.15) and (7.75,1.15) .. (p2);
    \draw[edge] (u10) .. controls (6.7,5.55) and (9.6,5.7) .. (p10);
    \draw[edge] (u3)  .. controls (6.85,3.45) and (8.25,1.05) .. (p6);

    \draw[edge] (u1)--(u2)--(u3);
    \draw[edge] (u1)--(u4);
    \draw[edge] (u1)--(u6);
    \draw[edge] (u1)--(u12);
    \draw[edge] (u2)--(u7);
    \draw[edge] (u3)--(u10);
    \draw[edge] (u4)--(u9);
    \draw[edge] (u4)--(u11);
    \draw[edge] (u5)--(u6)--(u7)--(u8)--(u9);
    \draw[edge] (u10)--(u11)--(u12)--(u5);
    \draw[edge] (u5)--(u9);
    \draw[edge] (u6)--(u10);
    \draw[edge] (u7)--(u11);
    \draw[edge] (u8)--(u12);

    \draw[edge] (p3)--(p2)--(p1)--(p4);
    \draw[edge] (p1)--(p6);
    \draw[edge] (p1)--(p12);
    \draw[edge] (p2)--(p7);
    \draw[edge] (p3)--(p10);
    \draw[edge] (p4)--(p9);
    \draw[edge] (p4)--(p11);
    \draw[edge] (p9)--(p8)--(p7)--(p6)--(p5)--(p12)--(p11)--(p10);
    \draw[edge] (p9)--(p5);
    \draw[edge] (p10)--(p6);
    \draw[edge] (p8)--(p12);
    \draw[edge] (p7)--(p11);
\end{scope}

\end{tikzpicture}%
}
\caption{The graph $Q$.}
\label{fig:Q}
\end{figure}

Let $Q$ be the graph as shown in Figure~\ref{fig:Q}. $Q$ was constructed in~\cite{Akbari2001} as an example that is uniquely 3-colorable but contains no $K_3$. Therefore, $P(Q,3)=3!=6$. Moreover, $\Pl(Q,3)$ can be determined by the following theorem.

\begin{theorem}[\cite{Akbari2006}]\label{them1}
Let $G$ be a graph on a set $V=\{v_1,v_2,\dots,v_n\}$ of $n\ge 1$ vertices. For $1\le i\le n$, let $L_{v_i}$ be a list of $d_i+1$ colors where $d_i\ge 0$ is a given integer. Suppose that $G$ is uniquely $L$-colorable and $d_1+d_2+\cdots+ d_n=m$, where $m$ is the size of $G$. Then $G$ has $f$-colorings for every list assignment $f$ of $G$ provided $f(v_i)=d_i+1$ for $1\le i\le n$.
\end{theorem}

\begin{proposition}\label{prop1}
$\Pl(Q,3)=P(Q,3)=6$.
\end{proposition}

\begin{proof}
Since $\Pl(Q,3)\le P(Q,3)=6$, we need only to show that $Q$ has at least 6 $L$-colorings for all 3-list assignments $L$ of $Q$.

Assume that $u,v$ are two vertices from different partition sets in the unique partition of $V(Q)$ into three independent sets. Let $L_0$ be a list assignment of $Q$ such that $L_0(u)=\{1\}, L_0(v)=\{1,2\}$, and $L_0(w)=\{1,2,3\}$ for all $w\in V(Q)\setminus\{u,v\}$. Then $Q$ is uniquely $L_0$-colorable as $Q$ is uniquely 3-colorable.
Also, 
note that $|V(Q)|=24$ and $|E(Q)|=45$. Then $L_0$ satisfies the requirements of Theorem~\ref{them1} as $0+1+2\times22=45$, which implies that
$Q$ has $L$-colorings for every list assignment $L$ of $G$ provided $|L(u)|=1$, $|L(v)|=2$, and $|L(w)|=3$ for all $w\in V(Q)\setminus\{u,v\}$.

Now we consider any 3-list assignment $L$ of $Q$. Assume that $L(u)=\{1,2,3\}$ and $L(v)=\{a,b,c\}$. For each $i\in [3]$, let $L_i(u)=\{i\}, L_i(v)=\{a,b\}$, and $L_i(w)=L(w)$ for all $w\in V(Q)\setminus\{u,v\}$. Then by Theorem~\ref{them1}, $Q$ has an $L_i$-coloring $\theta_i$, which is also an $L$-coloring. Further, let $L_i'(u)=\{i\}, L_i'(v)=L(v)\setminus \{\theta_i(v)\}$, and $L_i'(w)=L(w)$ for all $w\in V(Q)\setminus\{u,v\}$. We again obtain by Theorem~\ref{them1} an $L_i'$-coloring $\theta_i'$ for each $i\in [3]$, which is also an $L$-coloring and is different from $\theta_i$ as $\theta_i(v)\neq \theta_i'(v)$. Since $\theta_i(u)=\theta_i'(u)$ are pairwise distinct for all $i\in[3]$, we now have six $L$-colorings of $Q$, which completes the proof.
\end{proof}

For any pair of vertex-disjoint graphs $G$ and $H$, let $G\vee H$ be the \textit{join} of $G$ and $H$, i.e., $V(G\vee H)=V(G)\cup V(H)$ and $E(G\vee H)=E(G)\cup E(H)\cup \{uv: u\in V(G), v\in V(H)\}$.
Since $Q$ contains no $K_3$, $Q\vee K_{k-3}$ contains no $K_k$. We shall show that $Q\vee K_{k-3}$ is an ideal replacement for $G_{k,t}[X_k\cup S]$ after proving the next proposition.

\begin{lemma}[\cite{Kaul2018}]\label{kaul2018}
For any graph $G$ and $n,k\in N$, $$\Pl(G,k-n)P(K_n,k)\le \Pl(G\vee K_n, k)\le P(G\vee K_n, k)=P(G,k-n)P(K_n,k).$$
\end{lemma}

\begin{proposition}\label{prop3}
For any $k\ge 3$,
$\Pl(Q\vee K_{k-3},k)=P(Q\vee K_{k-3},k)=k!$.
\end{proposition}

\begin{proof}
By Lemma~\ref{kaul2018}, we have
$$\Pl(Q,3)P(K_{k-3},k)\le \Pl(Q\vee K_{k-3},k)\le P(Q\vee K_{k-3},k)=P(Q,3)P(K_{k-3},k).$$
Then by Proposition~\ref{prop1}, the first equality holds.

The second equality is then obvious as $Q\vee K_{k-3}$ is uniquely $k$-colorable. 
\end{proof}

Observe that $G_{k,t}[X_k\cup S]$ is uniquely $k$-colorable, where $x_1,x_2,\dots,x_{k}$ are in different color classes while $x_i$ and $s_i$ are in the same color class for $i\in[2]$, and $s_1,s_2$ are nonadjacent. We shall see that a similar structure also exists in $Q\vee K_{k-3}$.

Let $x_1',x_2',x_3',s_1',s_2'$ be the vertices in $Q$ as shown in Figure~\ref{fig:Q}, and let $V(K_{k-3})=\{x_4',x_5',\dots,x_{k}'\}$. 
Then $Q\vee K_{k-3}$ is uniquely $k$-colorable, where
$x_1',x_2',\dots,x_{k}'$ are in different color classes while $x_i'$ and $s_i'$ are in the same color class for $i\in[2]$, and $s_1',s_2'$ are nonadjacent. 
Thus replacing $G_{k,t}[X_k\cup S]$ and $x_1,x_2,\dots,x_{k},s_1,s_2$ in $G_{k,t}$  with $Q\vee K_{k-3}$ and $x_1',x_2',x_3',\dots,x_{k}',s_1',s_2'$ provides us another infinite family of graphs  supporting Theorem~\ref{main1} that contains no $K_{k}$. 
The proof can be carried out based on Proposition~\ref{prop3} by establishing a series of results analogous to Lemmas~\ref{prop:equality-at-k}, \ref{lem:ordinary-lower-bound}, \ref{lem:dominant-product}, and \ref{lem:list-upper-bound}, which together yield Proposition~\ref{thm:main}.

However, it remains unclear for which graphs the agreement of the chromatic polynomial and the list-color function always persists.

\begin{question}
Characterize the graphs $G$ such that for all integers $k$,
$P(G,k)=P_{\ell}(G,k)>0$
implies
$P(G,k+1)=P_{\ell}(G,k+1).$
\end{question}

\medskip
\noindent\textbf{Declaration on the use of AI.}
During the preparation of this work, the authors used AI systems to assist in
generating candidate examples for Theorem~\ref{main1}. All AI-generated suggestions were
verified and refined by the authors, who take full responsibility for the correctness
and originality of the paper.

\end{document}